\documentclass{amsart}
\usepackage{amssymb}
\usepackage{amsmath}
\usepackage{amsfonts}
\theoremstyle{plain}
\newtheorem{thm}{Theorem}[section]

\newtheorem{prb}[thm]{Problem}
\theoremstyle{remark}

\def\pmc#1{\setbox0=\hbox{#1}
    \kern-.1em\copy0\kern-\wd0
    \kern.1em\copy0\kern-\wd0}

\begin{document}
\title
[On noncontractible compacta]{On noncontractible compacta with trivial homology and homotopy groups}
\author[U.H. Karimov]{Umed H. Karimov}
\address{Institute of Mathematics,
Academy of Sciences of Tajikistan, Ul. Ainy $299^A$, Dushanbe
734063, Tajikistan} \email{umedkarimov@gmail.com}
\author[D. Repov\v{s}]{Du\v{s}an Repov\v{s}}
\address{Faculty of Mathematics and Physics, and
Faculty of Education,
University of Ljubljana, P.O.Box 2964,
Ljubljana 1001, Slovenia}
\email{dusan.repovs@guest.arnes.si}
\begin{abstract}
We construct an example of a Peano continuum $X$ such that: 
(i) $X$ is a one-point compactification of a polyhedron; 
(ii) $X$ is weakly homotopy equivalent to a point (i.e.
 $\pi_n(X)$ is trivial
for all $n \geq 0$); 
(iii) $X$ is noncontractible; and
(iv) $X$ is homologically and cohomologically locally connected (i.e. $X$
is a $HLC$ and  $clc$ space).
We also prove that 
all classical homology groups (singular,
\v{C}ech, and Borel-Moore), 
all classical cohomology groups (singular and \v{C}ech),
and 
all finite-dimensional Hawaiian groups of $X$ are trivial.
\end{abstract}
\dedicatory{ Dedicated to the memory of Professor Evgenij Grigor'evich Sklyarenko (1935-2009)} 
\date{\today}
\maketitle

\section{Introduction}

It is a
fundamental fact of homotopy theory that the existence of a
weak homotopy equivalence $f:K\to L$ between two $CW$-complexes
$K$ and $L$ implies that $f$ is actually a homotopy equivalence
($K\simeq_{w} L \Longrightarrow K\simeq L$).
Therefore if a
$CW$-complex $K$ has all homotopy groups trivial then $K$ is
necessarily contractible \cite{W}.

However, this is no longer true outside the class of
$CW$-complexes, e.g. the Warsaw circle $W$ is an example of a
planar noncontractible non-Peano continuum all of whose homotopy
groups are trivial (cf. e.g. \cite{N}). The failure of local connectivity of
$W$ is crucial, since it is well-known that every planar simply
connected Peano continuum must be contractible (cf. e.g. \cite{N}).

In our earlier paper \cite{KR} we constructed an example of a
noncontractible Peano continuum  with trivial homotopy groups. In
the present paper we shall construct in some sense sharper
example, namely a noncontractible Peano continuum $X$ which is a
one-point compactification of a polyhedron $P,$ which is
homologically locally connected (HLC space) and is weakly homotopy
equivalent to a point $X\simeq_{w} *$.

We shall also prove that all classical homology groups (singular,
\v{C}ech, and Borel-Moore), all classical cohomology groups (singular and \v{C}ech)
and all finite-dimensional Hawaiian earring groups of this space
$X$ are trivial. This answers our problem formulated in \cite{KR}.
We shall also state some new open problems.

\date{\today}

\section{Preliminaries}
We start by fixing some terminology and a notations which will be
used in the proof. All undefined terms can be found in
\cite{Br,H,KR,Kar,W}.

For any topological space $Z$ with a base point $z_0\in Z$
the {\it reduced} suspension $S(Z,z_0)$
is defined by
$$S(Z,z_0) = (Z\times I)/((Z\times
\{0\})\cup(Z\times \{1\})\cup(z_0\times I)),$$
where
$I$   is the
unit interval $I=[0,1]\subset \mathbb R$, and
the {\it unreduced} suspension $S'(Z)$
of $Z$ by
$$S'(Z) =(Z\times I)/((Z\times \{0\})\cup(Z\times \{1\})).$$

The {\it reduced} cone $C(Z,z_0)$ over $Z$ is defined by
$$C(Z,z_0) =(Z\times
I)/((Z\times \{1\})\cup(z_0\times I)),$$ and the {\it unreduced}
cone $C'(Z)$ over $Z$ by
$$C'(Z)=(Z\times I)/(Z\times \{1\}).$$

The $n-$dimensional Hawaiian earrings $(n = 0, 1, 2, \dots)$ is the
following subspace of the Euclidean $(n+1)-$space
$R^{n+1}$:
$${\mathcal{H}}^n = \{\bar x = (x_0, x_1, \dots x_n) \in \mathbb R^{n+1} |\ (x_0 - 1/k)^2 +
\Sigma_{i=1}^{n}x_i^2 = (1/k)^2,\ k\in { \mathbb N}\}.$$

\noindent In other words, ${\mathcal{H}}^n$ is a compact bouquet
of a countable number of $n-$dimensional spheres $S^n_k$ of radius
$1/k$. The point $\theta = (0, 0, \dots)$ is the base point of
${\mathcal{H}}^n.$ Obviously, $S({\mathcal{H}}^n, \theta) \cong
{\mathcal{H}}^{n+1}$ and $\pi_{n+1}(S({\mathcal{H}}^n, \theta))$
is an uncountable group, whereas $\pi_{n+1}(S'({\mathcal{H}}^n))$
is a countable group.

A {\it reduced} complex $\widetilde{K}_{\mathcal{C}}^*$-theory is
an extraordinary cohomology theory defined on the category of
pointed compacta and homotopic mappings with respect to base
points. For any compact pair of spaces $(X,A)$ with the base point
$x_{0}\in A$, $(X,A,x_0)$, there exists the following long exact
sequence (cf. e.g. \cite[p.55]{H}):

\begin{equation}\label{M-V}
\cdots \rightarrow \widetilde{K}_{\mathcal{C}}^n(X/A)\rightarrow
\widetilde{K}_{\mathcal{C}}^n(X)\rightarrow
\widetilde{K}_{\mathcal{C}}^n(A)\rightarrow
\widetilde{K}_{\mathcal{C}}^{n+1}(X/A)\rightarrow \cdots.
\end{equation}

We denote the homotopy classes of mappings with respect to the
base point by $[\ , \ ]$. On the category of connected spaces,
there exists a natural isomorphism of
cofunctors, for some $CW-$ complex $BU$ (cf. e.g. \cite[Theorem
1.32]{Kar}):
$$\widetilde{K}_{\mathcal{C}}^0(X) \cong [X, BU].$$

Every $CW-$complex is an
absolute neighborhood retract (ANR) therefore the functor
$\widetilde{K}_{\mathcal{C}}^0$ is  {\it continuous}, i.e. if
$X_i$ are compact spaces and $X = \underleftarrow\lim X_i,$ then

$$\widetilde{K}_{\mathcal{C}}^0(X) \cong
\underrightarrow\lim\widetilde{K}_{\mathcal{C}}^0(X_i).$$

\section{The construction of examples and the proofs of the main results}

Let $\mathcal P$ be an
inverse sequence of finite $CW$-complexes
$P_i$ :
$$P_0 \stackrel{f_0}{\longleftarrow} P_1
\stackrel{f_1}\longleftarrow P_2 \stackrel{f_2}\longleftarrow
\cdots.$$

Suppose that $P_0 = \{ p_0\}$ is a singleton and that all $P_i$
are {\it regular} finite $CW$-complexes, i.e. that they admit a finite
polyhedral structure. Let $C(f_0, f_1, f_2, \dots)$ be the
infinite mapping cylinder of $\mathcal P $ (cf. e.g. \cite{K, Si})
and let $\widetilde{\mathcal P}$ be its natural compactification
by the inverse limit $\underleftarrow{\lim} \mathcal{P}.$

Then the
space $\widetilde{\mathcal P}$ is an absolute retract (AR) (cf.
\cite{K}). Let $P^*$ be the quotient space of $\widetilde{\mathcal
P}$ by $\underleftarrow{\lim}\mathcal{P}$. Obviously, the space
$P^*$ is homeomorphic to the one-point compactification of the
countable polyhedron $C(f_0, f_1, f_2, \dots)$.

Let $C(f_n, f_{n+1}, \dots, f_m)$ be the finite cylinder of
mappings :
$$P_n \stackrel{f_n}{\longleftarrow} P_{n+1}
\stackrel{f_{n+1}}\longleftarrow \cdots P_m
\stackrel{f_m}\longleftarrow P_{m+1}.$$

Clearly, we may assume
that $$P_n\cup P_{m+1}\subset C(f_n,
f_{n+1}, \dots, f_m)\subset C(f_0, f_1, f_2, \dots).$$

We shall say that an inverse spectrum $\mathcal P$ is {\sl admissible} if the following conditions are satisfied:

\noindent (1) Every $P_i$ contains only one 0-dimensional cell
$p_i$ which is a base point and $f_i(p_{i+1}) = p_i$;

\noindent (2) No $CW$ complex $P_i$ contains any cells of positive
dimension  less than $i$; and

\noindent (3) For $i \geq 0$, every homomorphism
$\widetilde{K}_{\mathcal{C}}^0(f_i)$ is a nontrivial
isomorphism.

Such admissible spectra do exist. For example, the inverse spectra
constructed by Taylor \cite{T} or the suspension of the spectra
constructed by Kahn \cite{Kahn} are admissible (in our sense)
spectra.

\begin{thm}\label{Theorem}
Let $\mathcal{P}$ be an admissible spectrum. Then the one-point
compactification $P^*$ of the countable polyhedron $C(f_0, f_1,
f_2, \dots)$ has the following properties:

(i) $X$ is weakly homotopy equivalent to a point,
$P^*\simeq_{w}*$, (i.e.
 $\pi_n(X)$ is
trivial
for all $n \geq 0$);

(ii) $P^*$ is acyclic with respect to all classical
homology groups (singular, Borel-Moore, and \v{C}ech);

(iii) $P^*$ is acyclic with respect to all classical cohomology
groups (singular and \v{C}ech);

(iv) $P^*$ is a homologically  locally connected ($HLC$);

(v) $P^*$ is a cohomologically  locally connected ($clc$); and

(vi) $P^*$ is a noncontractible Peano continuum.

\end{thm}

\begin{proof} By equality (\ref{M-V}), we have
the
following natural exact sequence:
$$\cdots \rightarrow
\widetilde{K}_{\mathcal{C}}^{0}(\widetilde{\mathcal{P}})\rightarrow
\widetilde{K}_{\mathcal{C}}^{0}(\underleftarrow{\lim}
\mathcal{P})\rightarrow
\widetilde{K}_{\mathcal{C}}^1(P^*)\rightarrow
\widetilde{K}_{\mathcal{C}}^1(\widetilde{\mathcal{P}})\rightarrow\cdots.$$
As it was mentioned before,   $\widetilde{\mathcal{P}}$ is an absolute retract,
therefore
$\widetilde{K}_{\mathcal{C}}^n(\widetilde{\mathcal{P}})= 0$.
On the other hand, the
group $\widetilde{K}_{\mathcal{C}}^0(\underleftarrow{\lim}
\mathcal{P})$ is nontrivial since the cofunctor
$\widetilde{K}_{\mathcal{C}}^0$ is continuous and all
homomorphisms $\widetilde{K}_{\mathcal{C}}^0(f_i)$ are nontrivial
isomorphisms.
It follows by exactness of the sequence (\ref{M-V})  that
the group
${\widetilde{K}}_{\mathcal{C}}^1(P^*)$ is also nontrivial and hence
the space
$P^*$ must be
noncontractible, as asserted.

Let us prove that all homotopy groups $\pi_{n\ge 1}(P^*)$ are
trivial. Fix a number $n\in \mathbb N.$ Let $m\in \mathbb N$ be
any number such that $m > n.$ Consider a relative $CW-$complex
$(P^*,C(f_m, f_{m+1}, \dots )^*)$ with the compactification point
$\ast$ as the base point. Let $f$ be any mapping of the sphere
$S^n$ with some base point $pt$ to $P^*$ i.e. we have a mapping
$f$ of the relative $CW-$complex $(S^n,pt)$ to $(P^*,C(f_m,
f_{m+1}, \dots )^*).$

By the Cellular Approximation Theorem (cf. e.g. \cite{W}) the
mapping $f$ is homotopic relative $pt$ to a cellular map
$$g_m:(S^n,pt)\to (P^*,C(f_m, f_{m+1}, \dots )^*).$$ Note that the
space $P^*$ can be represented as the union of the following three
spaces:
$$P^* = C(f_0, f_{1}, \dots, f_{m-1})\cup C(f_{m})\cup C(f_{m+1}, \dots
)^*.$$

Observe that $CW-$complex $C(f_{m})$ consists of two 0-dimensional
cells, one 1-dimensional cell and some cells of dimension larger
than $n$, since $m > n$ by our choice of the number $m.$ Since the
mapping $g_m$ is a cellular mapping of the pairs it follows that
the image of $g_m$ lies in the union of the $n-$dimensional
skeleta of $P^*$ and $C(f_{m+1}, \dots )^*.$

Since $C(f_{m})$
contains only one 1-dimensional cell $e^1$ and cells of dimension
larger than $n$ we have
$$\textrm{Im}(g_m)\subset C(f_0, f_{1}, \dots, f_{m-1})\cup e^1\cup C(f_{m+1}, \dots )^*.$$

The space $\textrm{Im}(g_m)\subset C(f_0, f_{1}, \dots,
f_{m-1})\cup e^1$ is contractible with respect to the point
$e^1\cap C(f_{m+1}, \dots )^*$ therefore we may assume that the
mapping $f$ is  homotopic, with respect to the subspace
$C(f_{m+1}, \dots )^*$, to the mapping $g_m$, the image of which
lies in $C(f_{m+1}, \dots )^*$.

Since the relative $CW-$complexes $(C(f_k, f_{k+1}, \dots )^*,
C(f_{k+1}, f_{k+2}, \dots )^*)$ for
$k>m$ contain only one
1-dimensional cell plus only cells of the dimensions larger than
$n$, it follows that $g_m$ (and therefore $f$) is homotopic
 relative to $pt$, to the constant mapping to the
point $\ast$. Therefore $\pi_n(P^*, \ast) = 0.$

The property of homological local connectedness with respect to
singular homology $\textrm{HLC}$ at all points, except at
the base
point, follows by the fact that the space $C(f_0, f_{1}, f_{2},
\dots)$ being a $CW-$ complex, is always locally contractible. As
we have seen, it easy to show that $\pi_n(C(f_m, f_{m+1}, f_{m+2},
\dots)^{\ast},\ast) = 0$. It now follows by the Hurewicz Theorem
that singular homology groups are trivial:
$$H_n(C(f_m, f_{m+1},
f_{m+2}, \dots)^{\ast}, \ast) = 0, \ \ \ \hbox{for} \ \ n < m.$$
Hence the space
$P^*$ is also an $\textrm{HLC}$ space at the base point. The
 $\textrm{clc}$ property
 follows from $\textrm{HLC}$ (cf.
\cite{Br}).

Finally, let us verify the acyclicity of $P^*.$
Since $\pi_n(P^*,
\ast) = 0$ for all $n \geq 0$,
it follows by the Hurewicz Theorem
that
the singular homology groups of $P^*$ are trivial for all $n,$
$H_n(P^*,
\ast) = 0$.
It is well-known that all classical homology are naturally isomorphic
on the category of compact metrizable $\textrm{HLC}$
spaces (cf. e.g.
\cite{Br}).
Therefore the space $P^*$ is acyclic in \v{C}ech, Borel-Moore, and
Vietoris homology theories.
By invoking the  Universal
Coefficients Theorem, we can conclude that over $\mathbb Z$, all
singular, \v{C}ech, Alexander-Spanier, and
sheaf cohomology groups of the space
$P^*$ are trivial, too.
\end{proof}

\section{The infinite-dimensional Hawaiian earrings and the
infinite-dimensional Hawaiian group}

The $n-$dimensional {\it Hawaiian set} of a space $X$, $n\in\{0,1,2,...\},$
with the base
point $x_0\in X$ is defined as set of
all homotopy classes $[f]$ of the
mappings
$$f:({\mathcal{H}}^n, \theta)\to (X, x_0)$$ of the
$n$-dimensional Hawaiian earrings
${\mathcal{H}}^n$
into $X$.
For $n\ge 1$
there is a natural multiplication with respect to which this set
is a group. We denote it by ${\mathcal H}_n(X, x_0)$ and call it
the {\it Hawaiian group} in dimension $n$ 
(cf. \cite{KR}).

The Hawaiian groups $\mathcal H_n(X, x_0)$ (the set $\mathcal
H_0(X, x_0)$) are homotopy invariant in the category of all
topological spaces with base points and continuous mappings.
Note
that for the cone over the 1-dimensional Hawaiian earrings the
group ${\mathcal{H}}_1(C({\mathcal{H}}^1), pt)$ is nontrivial, for
some points $pt$ 
(cf. \cite{KR}).

The space $P^*$ is locally contractible at all points except the
base
point,
and for every natural number $n$ there exists a
neighborhood
$U_{\ast}$ such that $\pi_n(U_{\ast})$ is trivial.
Therefore (since $\pi_n(P^*)=0$) it follows that ${\mathcal{H}}_n(P^*)=0.$

Consider a compact bouquet of Hawaiian earrings of all dimensions
${\mathcal{H}^{\infty}} = \bigvee_{n=1}^{\infty}{\mathcal{H}}^n$
with respect to their base points. Call the space ${\mathcal{H}}$
the {\sl infinite-dimensional Hawaiian earrings}. There is a
natural base point $pt.$
We shall call
the set of all homotopy classes of maps
$[({\mathcal{H}}^{\infty},pt),
(X,x_0)]$ with the natural multiplication  the
infinite-dimensional Hawaiian group
${\mathcal{H}_{\infty}}(X,x_0),$ where $x_0$ is the base point of
the space $X.$

\begin{thm}\label{Theorem 2}
The infinite-dimensional Hawaiian group of the spaces $P^*$
constructed by the admissible spectra of Taylor is nontrivial,
${\mathcal{H}_{\infty}}(P^*,\ast)\neq 0$.
\end{thm}
\begin{proof}

The inverse spectrum of Taylor can be described as follows. Let
$M$ be the Moore space $ = S^{2q-1}\cup_p e^{2q}, p\geq 3.$ Toda
bracket gives the mapping $f:S^{2(p-1)}(M)\rightarrow M$
of the
$2(p-1)-$fold suspension of
$M$ to
$M.$
Let the space $P_1$ be the
singleton
$\{p_1\}$, $P_2 = M$, $P_{i+2} = S^{2(p-1)i}(M)$ and
$f_2 = f,$
$f_{i+1} = S^{2(p-1)}(f_i)$.
Then we get the desired inverse spectrum.

According to Adams and Toda we have mappings $\phi$ and $\psi_i$
such that the composition $\phi f_2f_3\cdots f_i\psi_i$ is a
nontrivial Toda's element $\alpha_i$ (cf. \cite[Pages 12-13]{Adams
alpha}, \cite{Toda alpha}):

\begin{eqnarray}
{P_1}\stackrel{f_1}\leftarrow{P_2}&\stackrel{f_2}\leftarrow
P_3\stackrel{f_3}\leftarrow\cdots \stackrel{f_{i+1}}\leftarrow
               &P_{i+2}\leftarrow\cdots \\ \nonumber
\varphi\downarrow&  &\uparrow\psi_i \\  \nonumber
{S^{2q}}&
&S^{2q-1+2(p-1)i} \\   \nonumber
\end{eqnarray}

Define the mapping $f:{\mathcal{H}^{\infty}}\to P^*$ as follows.
Consider the compact bouquet of spheres
$\bigvee_{i=1}^{\infty}S^{2q-1+2(p-1)i}$. On every sphere
$S^{2q-1+2(p-1)i}$ we have a mapping $\psi_i$ to $P_{i+2}.$
The set of
all mappings $\{\psi_i\}$ naturally generates the mapping of
$\bigvee_{i=1}^{\infty}S^{2q-1+2(p-1)i}$ to $P^*.$
The space
$\bigvee_{i=1}^{\infty}S^{2q-1+2(p-1)i}$ can be considered as a
subspace of ${\mathcal{H}^{\infty}}$.

Let now $f$ be the extension of this
mapping to the entire
space ${\mathcal{H}^{\infty}}$ which maps the
complement
to $\ast.$ The mapping $f$ is an
essential mapping.

Indeed, suppose that $f$ were inessential. Then we would have a
homotopy $H:{\mathcal{H}^{\infty}}\times [0,1]\to P^*$ such that
$H(\theta, 0)= \ast.$ The restrictions of $H$ on every sphere
$S^{2q-1+2(p-1)i}$ would be inessential in the space $P^\ast
\setminus \{p_1\}$  for large $i$.

For simplicity we shall again denote these
restrictions  by $H.$ So we have for a large $i$ the homotopy
$$H:S^{2q-1+2(p-1)i}\times [0,1]\to P^*,$$
connecting the mapping
$\psi_i$ and the constant mapping in $P^\ast \setminus \{p_1\}.$ The
$CW$ complex $S^{2q-1+2(p-1)i}\times [0,1]$ is an
$(2q+2(p-1)i)$-dimensional complex.

Choose an integer $m > 2q+2(p-1)i$ and
consider the relative $CW$-complex
$(P^*,C(f_{m+1},f_{m+2},\cdots)^*).$
By the Cellular Approximation Theorem
the mapping
$$H:S^{2q-1+2(p-1)i}\times [0,1]\to P^*\setminus \{p_1\}$$
is
homotopy equivalent to a
cellular map, with respect to the set
$S^{2q-1+2(p-1)i}\times \{1\}.$

Since the complex $C(f_{m+1})$
contains only two 0-cells, one 1-cell, plus
cells of dimension larger
than $2q+2(p-1)i$, we may assume
that the image of the
homotopy
$$H:S^{2q-1+2(p-1)i}\times [0,1]\to P^*\setminus \{p_1\}$$
lies in the space
$$C(f_{1}, \dots, f_{m-1})\cup e^1 \cup C(f_{m+1}, \dots
)^*\setminus \{p_1\}.$$

Now, the space $P_2$ is a retract of this space. So we have a
mapping of the sphere $S^{2q-1+2(p-1)i}$ to the sphere $S^{2q}$,
which should be inessential. But this  contradicts  the
nontriviality of the Toda element mentioned above. Therefore
${\mathcal{H}_{\infty}}(P^*,\ast)\neq 0.$

\end{proof}

\section{Epilogue}

Since our example in \cite{KR} is infinite-dimensional, it is
natural to ask the following question \cite{EKR}:

\begin{prb}
Does there exist a finite-dimensional noncontractible Peano
continuum all homotopy groups of which are trivial?
\end{prb}

Our Theorem {\ref{Theorem}} gives an answer to our problem from
\cite{KR} but the cases of finite-dimensional spaces and
infinite-dimensional Hawaiian groups remain open:

\begin{prb}
Let $P$ (resp. $P^{*}$) be the
one-point compactification of
any  finite-dimensional countable polyhedron
by the point  $\theta\in P$ (resp. $\theta^* \in P^{*}$).
Suppose that $f:(P,\theta)\to (P^*, \theta^*)$
is any continuous mapping such that
$$\mathcal{H}_n(f):\mathcal{H}_n(P, \theta)\to \mathcal{H}_n(P^*,\theta^*)$$
 is an isomorphism for every $n\in \mathbb N$.
Is then $f$ a homotopy equivalence?
\end{prb}

\begin{prb} Let $P$ (resp. $P^{*}$) be the
one-point compactification of
a connected polyhedron
by the point  $\theta\in P$ (resp. $\theta^* \in P^{*}$).
Suppose that  $f:(P,\theta)\to (P^*, \theta^*)$
is a continuous mapping such that
$$\mathcal{H}_{\infty}(f):\mathcal{H}_{\infty}(P, \theta)\to \mathcal{H}_{\infty}(P^*,\theta^*)$$
is an isomorphism.
Is then $f$ a homotopy equivalence?
\end{prb}

\section{Acknowledgements}
This research was supported by the Slovenian Research Agency
grants P1-0292-0101, J1-9643-0101 and J1-2057-0101.
We thank the referee for comments and suggestions.

\providecommand{\bysame}{\leavevmode\hbox to3em{\hrulefill}\thinspace}

\end{document}